\documentclass[11pt]{amsart}

\usepackage{amsmath,amsthm, amssymb}
\usepackage[all]{xy}
\usepackage{enumerate}
\usepackage{wasysym}
\usepackage{graphicx,color}

\newcommand{\lap}{\Delta}
\newcommand{\lie}{\mathcal{L}}

\newcommand{\Id}{I}
\newcommand{\pr}{\Pi_\lap}
\newcommand{\prprime}{\Pi_{\lap'}}
\newcommand{\lef}{L\!e\!f}
\newcommand{\rel}{\mathcal{R}}
\newcommand{\aux}{\mathcal{A}}
\newcommand{\eps}{\epsilon}
\newcommand{\epseta}{\eps_\eta}
\newcommand{\ixi}{i_\xi}
\newcommand{\gb}{\rule{0.1cm}{0.1cm}}
\newcommand{\gc}{\bullet}

\DeclareMathOperator{\ad}{ad}

\newtheorem{theorem}{Theorem}[section]
\newtheorem{lemma}[theorem]{Lemma}
\newtheorem{corollary}[theorem]{Corollary}
\newtheorem{proposition}[theorem]{Proposition}

\theoremstyle{definition}
\newtheorem{definition}[theorem]{Definition}

\theoremstyle{remark}

\numberwithin{equation}{section}
\title{Hard Lefschetz theorem for Sasakian manifolds}

\author[B. Cappelletti-Montano]{Beniamino Cappelletti-Montano}
 \address{Dipartimento di Matematica e Informatica, Universit\`a degli Studi di
 Ca\-gli\-ari, Via Ospedale 72, 09124 Cagliari, Italy}
 \email{b.cappellettimontano@gmail.com}

\author[A. De Nicola]{Antonio De Nicola}
 \address{CMUC, Department of Mathematics, University of Coimbra, 3001-501 Coimbra, Portugal}
 \email{antondenicola@gmail.com}

\author[I. Yudin]{Ivan~Yudin}
 \address{CMUC, Department of Mathematics, University of Coimbra, 3001-501 Coimbra, Portugal}
 \email{yudin@mat.uc.pt}

\thanks{Research partially supported by CMUC and FCT (Portugal), through
European program COMPETE/FEDER, grants PEst-C/MAT/UI0324/2011, PTDC/MAT/099880/2008 and
MTM2009-13383 (A.D.N.),  SFRH/BPD/31788/2006 (I.Y.), Prin 2010/11
-- Variet\`a  reali e complesse: geometria, topologia e analisi armonica -- Italy (B.C.M.)}

\begin{document}

\begin{abstract}
We prove that on a compact Sasakian manifold $\left( M,\eta, g \right)$ of dimension $2n+1$, for any $0\le p\le n$ the wedge product with $\eta\wedge(d\eta)^p$ defines an isomorphism between the spaces of harmonic forms $\Omega^{n-p}_\lap\left( M \right)$ and  $\Omega^{n+p+1}_\lap\left( M \right)$. Therefore it induces an isomorphism between the de Rham cohomology spaces $H^{n-p}(M)$ and  $H^{n+p+1}(M)$. Such isomorphism is proven to be independent of the choice of a compatible Sasakian metric on a given contact manifold. As a consequence, an obstruction for a contact manifold to admit Sasakian structures is found.
\end{abstract}

\maketitle

\section{Introduction}

Sasakian manifolds, introduced by Sasaki \cite{sasaki} in 1960, can be described as an odd-dimensional counterpart of K\"{a}hler manifolds.
Starting from the 90s, a renewed interest in Sasakian geometry was stimulated by
new findings in theoretical physics (see e.g. \cite{gauntlett2004, kunduri2012,
martelli2006, martelli2008}), especially after the Maldacena conjecture
\cite{maldacena1998} on the duality between conformal field theory and
supergravity on anti-de-Sitter space-time. As a consequence, many important geometric and topological properties of Sasakian manifolds were discovered, see e.g. \cite{galicki2005, galicki2008, futaki2009, kollar2007}.

It is well known that Sasakian geometry is naturally related to
K\"{a}hler geometry from two sides: on the one hand Sasakian manifolds can be defined as those manifolds whose metric cone is K\"{a}hler, on the other hand the $1$-dimensional foliation defined by the Reeb vector field  is transversely K\"{a}hler.

A remarkable property of compact K\"{a}hler manifolds is given by the celebrated
Hard Lefschetz Theorem, stating that the cup product with the suitable powers of
the symplectic form gives isomorphisms between the de Rham cohomology groups of
complementary degrees. This result was first stated by Lefschetz in
\cite{lefschetz} but the first complete proof was given by Hodge in
\cite{hodge1952}. One of the applications of the Hard Lefschetz Theorem is
that it gives an obstruction for the existence of a K\"ahler metric on a
compact symplectic manifold.

Later on, an odd dimensional version of the Hard Lefschetz Theorem was proven
for compact coK\"ahler manifolds in \cite{marrero1993}.
Thus one may ask whether a similar property also holds in the context of Sasakian geometry.
So far, the only result that was obtained in this direction is the transversal
Hard Lefschetz Theorem, proved by El Kacimi-Alaoui in \cite{elkacimi}, which holds for the basic cohomology with respect to any homologically orientable transversely K\"{a}hler foliation. In this paper, our aim is to prove that a version of Hard Lefschetz Theorem holds for the de Rham cohomology
of any compact Sasakian manifold.

Our approach needs to be different from the one adopted in  K\"{a}hler geometry.
Indeed, although Sasakian and K\"{a}hler manifolds share many properties, in
this case the picture for Sasakian manifolds shows profound peculiarities.
  Let $\left( M,\omega,g \right)$ be a compact K\"{a}hler
manifold of dimension $2n$ and $\Omega^p_\lap\left( M \right)$ the space of harmonic $p$-forms on $M$. We recall that the Hard Lefschetz Theorem
states that for any $0\le p\le n$, the maps
\begin{align*}
    \omega^p \wedge -\colon \Omega^{n-p}_\lap \left( M \right) & \to
    \Omega^{n+p}_\lap\left(
    M \right)\\[1ex]
    \alpha &\mapsto \omega^p\wedge \alpha
\end{align*}
are isomorphisms.  Now, let $\left( M,\eta, g \right)$ be a compact Sasakian manifold of dimension $2n+1$. As a natural generalization of the above Lefschetz
isomorphism, one can consider, for each $0\le p\le n$, the maps
\begin{equation}\label{sasakian_lef_map}
\begin{aligned}
    \eta\wedge (d\eta)^{p}\wedge -\colon \Omega^{n-p}_\lap\left( M \right) & \to
    \Omega^{n+p+1}_\lap\left( M \right)\\[1ex]
    \alpha &\mapsto \eta\wedge (d\eta)^p \wedge\alpha.
\end{aligned}
\end{equation}
However at this step a serious problem already arises. Namely, differently from the K\"{a}hler case, it is not true that the wedge multiplication by
either $d\eta$ or $\eta\wedge d\eta$ maps harmonic forms into harmonic forms, so in principle the definition of the above maps may happen to be ill posed.

Nevertheless, we discover some spectral properties of the Laplace operator on differential forms which allow to overcome
this obstacle. More precisely, given $\alpha\in \Omega^{n-p}_\lap\left( M \right)$, we show that the forms $\eta\wedge (d\eta)^k\wedge \alpha$ and
$(d\eta)^{k-1}\wedge \alpha$ are eigenforms of the Laplacian with positive integer eigenvalues for all $0\le k\le p-1$. These eigenforms and their
eigenvalues are visualized in Figure~1 at page~\pageref{plot} in the case $n=5$.
In the figure, the points on the horizontal axis represent the spaces of
harmonic forms on~$M$. All other points with coordinates $\left( p,\nu \right)$ denote suitable subspaces of
$$
\left\{\, \beta\in \Omega^p\left( M \right) \,\middle|\, \lap \beta = 4\nu \beta
\right\}.
$$
The segments represent the isomorphisms~\eqref{isomorphismddelta} and~\eqref{isomorphismlambdal} between the corresponding vector spaces.  As shown in the
figure, the  $\lap$-eigenvalues of $\eta\wedge(d\eta)^k\wedge \alpha$ initially increase with $k$ up to degree $n$ and then decrease until to reach zero, for $k=p$. Thus each mapping in \eqref{sasakian_lef_map} is actually well defined, in the sense that its target space is
$\Omega_{\lap}^{n+p+1}(M)$. Moreover, in Theorem \ref{lefschetz} we prove that such map is in fact an isomorphism by
explicitly giving its inverse map.
This isomorphism obviously induces via Hodge theory an isomorphism between the
corresponding de Rham cohomology groups. Namely, for each $0\le p\le n$ we obtain the isomorphism
\begin{equation*}
	\begin{aligned}
		\lef_{n-p}\colon 	H^{n-p}(M) &\to H^{n+p+1}(M)\\
	\left[\, \beta \,\right] &\mapsto \left[\, \eta\wedge (d\eta)^{p}\wedge\pr\,  \beta \,\right],
\end{aligned}
\end{equation*}
where $\pr\beta$ denotes the orthogonal projection of $\beta$ on the space of harmonic forms.
Note that, contrary to the symplectic case, we are forced to use the metric structure in the definition of  $\lef_p$.
Thus, \emph{a priori}, one could expect that different Sasakian metrics on $\left(
M,\eta
\right)$ could lead to different Lefschetz isomorphisms.
To the utter surprise of the authors, this is not the case. In
Theorem~\ref{main2} we prove that the Lefschetz isomorphism is independent of
the metric. This provides an obstruction for a contact manifold to admit
Sasakian structures. In the last section, we introduce the notion of Lefschetz
contact manifold and we prove that their odd Betti numbers up to the middle
dimension are even.

\section{Preliminaries}
In this section we recall the definition of Sasakian manifolds and list some of
their properties. For further details we refer the reader to
\cite{blairbook2010, galicki-book}.

Let $M$ be a smooth manifold of dimension $2n+1$. A $1$-form $\eta$ on $M$ is called a
\emph{contact form} if $\eta\wedge d\eta^n$ nowhere vanishes.
In this case the pair $(M,\eta)$ is called a (strict) \emph{contact manifold}.
We write $\Phi$ for $\frac12 d\eta$ and we denote by $\xi$ the  \emph{Reeb
vector field} on $\left( M,\eta \right)$,
that is the unique vector field on $M$ such that $i_{\xi}\eta=1$ and $i_{\xi}d\eta=0$.

Let $(M,\eta)$ be a contact manifold and $g$ a Riemannian metric on $M$.
 We
define the endomorphism $\phi\colon TM \to
TM$ by $ \Phi (X, Y ) = g( X, \phi Y)$.

Then $(M, \eta ,g)$ is called a \emph{Sasakian manifold} if the following conditions hold:
\begin{enumerate}[($i$)]
\item $\phi^2 =- \Id + \eta \otimes \xi$, where $\Id$ is the identity operator;
\item $g(\phi X , \phi Y )= g(X,Y) - \eta(X)\eta (Y)$ for any vector fields
	$X$ and $Y$ on $M$;
\item the normality condition is satisfied, namely
		\begin{equation*}
			 \left[ \phi, \phi \right]_{FN} +  2 d\eta \otimes
			\xi = 0,
		\end{equation*}
            where $[-,-]_{FN}$ is the Fr\"olicher-Nijenhuis
	    bracket (cf.~\cite{michor}).
\end{enumerate}

It is well known that in any Sasakian manifold
\begin{align}
    \label{etaphi}
    \phi \xi &=0, & \eta \circ \phi &=0.
\end{align}

Now we will introduce notation for some linear operators on the de~Rham algebra $\Omega^*(M)$.
For a $p$-form $\alpha$ on $M$, we denote by $\eps_\alpha$ the operator given by
\begin{equation*}
	\eps_\alpha \beta = \alpha \wedge \beta,
\end{equation*}
where $\beta\in \Omega^*(M)$.
If $\left(M,\eta \right)$ is a compact contact manifold, then $\epseta$ is adjoint to $\ixi$ with respect to the usual
global
scalar product on $\Omega^*(M)$, that is $\epseta = \ixi^*$.
Since $\ixi \eta = 1$, for any $\omega\in \Omega^*(M)$, we get
\begin{align}
     \left\{ \ixi, \epseta\right\}\omega = \omega,
    \label{lambdal}
\end{align}
where the curly brackets are used to  denote the anti-commutator of two operators.

Let $(M,\eta,g)$ be a compact Sasakian manifold. We define the operators $L$ and $\Lambda$ on $\Omega^*(M)$ by
\begin{align*}
L & = \eps_\Phi, & 	 \Lambda = L^*.
\end{align*}
Then, since $d$ is a graded derivation on $\Omega^*(M)$ of degree $1$ and
$\eta$ is a $1$-form, we get
\begin{align}
    \label{dl}
    \left\{ d, \epseta \right\} = \eps_{d\eta} = 2\eps_\Phi =2L.
\end{align}
Therefore
\begin{align}
    \label{deltalambda}
    \{i_\xi, \delta\} = \{d, \epseta\}^* = 2 L^* = 2\Lambda.
\end{align}
Hereafter, we will use some elements of
Fr\"olicher-Nijenhuis calculus, developed in~\cite{indag} (see also
\cite[Section~8]{michor}).
In this framework, for every
vector valued $(k+1)$-form $\psi$, the graded derivation $i_\psi$
of degree $k$
on
$\Omega^*(M)$  is defined. It acts on $\omega\in \Omega^p(M)$ by
\begin{equation*}
	\left( i_\psi \omega \right)( X_1,\dots, X_{p+k} ) =
	\sum_{\sigma} (-1)^\sigma \omega ( \psi( X_{\sigma(1)},\dots,
	X_{\sigma(k+1)}),  \dots, X_{\sigma(p+k)}
	),
\end{equation*}
where  the summation is taken over all $(k+1,p-1)$-shuffles.
In the case $k+1=0$, that is when $\psi=X$ is a vector field, we reobtain the usual
interior product $i_X$.
If $\psi$ is an endomorphism of $TM$, that is when $k=0$, the formula above can
be rewritten as
\begin{equation*}
	i_\psi \omega \left( X_1, \dots, X_p \right) = \sum_{s=1}^p \omega\left( X_1,
\dots, \psi X_s, \dots, X_p
\right).
\end{equation*}
For any vector valued $(k+1)$-form $\psi$
 the operator $\lie_\psi$ is
 defined to be  $ i_\psi d - (-1)^{k} di_\psi$.
Note that every $\lie_\psi$ is a graded derivation of degree $k+1$ on
$\Omega^*(M)$.
If $\psi=X$ is a vector field, then $\lie_X$ is the usual Lie derivative.

When $\psi$ is the identity operator $\Id\colon TM\to TM$, the operator $i_\psi$ will be denoted by
$\deg$, motivated by the fact that for any $p$-form $\omega$ we have
\begin{equation*}
	\deg \omega = i_\Id \omega = p\omega.
\end{equation*}
Now we summarize several results from~\cite{fujitani} on commutators between
operators on $\Omega^*\left( M \right)$, where $M$ is a  Sasakian manifold
(N.B.: Fujitani uses $\Phi$ for $i_\phi$, $\varphi$ for
$\Phi$, $\lambda$ for $\ixi$, and $l$ for $\epseta$).

\begin{theorem}
    \label{commutators}
    Let $M$ be a compact Sasakian manifold of dimension $2n + 1$. Then
         the operator $i_\phi$ commutes with $\epseta$, $\ixi$,
            $L$, and $\Lambda$.
         The Lie derivative $\mathcal{L}_\xi$ commutes with
            $d$, $\delta$, $\epseta$, $\ixi$, $L$, $\Lambda$, and
            $i_\phi$.
            Furthermore
            \begin{align}
                \label{dLambda}
\left[ d,\Lambda \right] &= \left[ i_\phi, \delta \right] -2\left( n-\deg
\right)\ixi\\
\label{laplambda}
\left[ \lap, \ixi \right] & = 2\left[ i_\phi, \delta \right] - 4\left( n
-\deg
\right)\ixi\\
\label{lapiphi}
\left[ \lap, i_\phi \right] & = -2 \left( \mathcal{L}_\xi - \ixi d + \epseta \delta
\right)\\
\label{lapl}
\left[ \lap, \epseta \right] & =- 2 \lie_\phi + 4 \epseta\left( n -\deg
\right).
    \end{align}
\end{theorem}
\begin{proof}
    See Propositions~1.1,~1.2,~3.3 and Theorem~3.2 in~\cite{fujitani}.
\end{proof}
Using the equalities \eqref{dLambda}--\eqref{lapl}, Fujitani reobtained a few results of
Tachibana~\cite{tachibana} in~\cite[Theorem~4.1]{fujitani} and complemented them with the dual ones
 in
\cite[Corollary~4.2]{fujitani}. Below is the summary of these results.
\begin{proposition}
    \label{tachi}
    Let $M$ be a compact Sasakian manifold of dimension $2n+1$ and
    $\omega$ a harmonic $p$-form on $M$. Then $i_\phi \omega$ is  also harmonic and,
    moreover,
    \begin{enumerate}[(i)]
        \item if $p\le n$, then $\ixi \omega =0$;
        \item if $p\ge n+1$, then $\epseta\omega =0$;
        \item if $p\le n+1$, then $\Lambda \omega = 0$;
        \item if $p\ge n$, then $L \omega =0$.
    \end{enumerate}
\end{proposition}
In the following proposition we prove some other useful identities.
\begin{proposition}
    In any Sasakian manifold we have
    \begin{align}
        \label{iphil}
     &   \left[ i_\phi ,\epseta \right]  = 0 \\
     \label{iphilambda} & \left[ i_\phi,\ixi \right]    = 0\\
	\label{liephisquare}& \lie_\phi^2 = -2L\lie_\xi\\
    \label{deltal}
  &   \left\{ \delta, \epseta \right\} = -\lie_\xi
.
    \end{align}
\end{proposition}
\begin{proof}
     Since $i_\phi$ is a derivation of degree zero, we have
     $\left[ i_\phi, \epseta \right] = \epsilon_{i_\phi \eta}=0$, as $i_\phi\eta =0$
     by~\eqref{etaphi}.
Next, it is easy to check that for any $\psi\colon TM\to TM$ and
any vector field~$X$ we have $\left[ i_\psi , i_X \right] = -i_{\psi X}$. Thus
$\left[i_\phi, \ixi  \right]$ is zero by~\eqref{etaphi}.
     Further, according to Fr\"olicher-Nijenhuis calculus, we have $\left\{ \lie_\phi, \lie_\phi
     \right\} = \lie_{ \left[ \phi,\phi \right]_{FN}}$. Thus, from the
   normality condition for Sasakian manifolds, we get
\begin{align*}
	\lie_\phi^2&
	= \frac12
	\lie_{\left[ \phi,\phi \right]_{FN}} = -\lie_{d\eta\otimes \xi}
	 =-\left\{ i_{d\eta\otimes \xi}, d \right\} = -\left\{ d\eta \wedge
	i_\xi, d \right\}\\& =\!  -
	d\eta \wedge \lie_\xi = - 2L\lie_\xi.
\end{align*}
Finally, it was shown
on page 109 of~\cite{goldberg}
that, for any Killing vector field $X$, one has
$
\lie_X + \left\{ \delta, \eps_{g(X,-)}  \right\} = 0
$.
Since in any Sasakian manifold the Reeb vector field $\xi$ is Killing and
$\eta=g\left( \xi,- \right)$,
the equation
\eqref{deltal} holds.
\end{proof}

\section{Hard Lefschetz isomorphism for harmonic forms}
In this section we establish the Hard Lefschetz isomorphism   between the spaces of
harmonic $p$-forms and harmonic $(2n+1-p)$-forms in a compact Sasakian manifold
$M$  of dimension $2n+1$.

We start by introducing some notation.
Let $A_1$, \dots, $A_k$ be operators on $\Omega^*\left( M \right)$.  We will denote by
$\Omega^p_{A_1,\dots, A_k}\left( M \right)$ the set of $p$-forms $\omega$ such that
$A_1\omega = \dots = A_k\omega =
0$ and we will use $\Omega^{p,\nu}_{A_1,\dots, A_k}\left( M \right)$ as a shorthand for
$\Omega^p_{\lap-\nu\Id, A_1,\dots, A_k}\left( M \right)$, where $\nu$ is a real
number. Of course these spaces are empty if $\nu<0$ and $M$ is compact, since all $\lap$-eigenvalues
are non-negative in this case.

We will be mainly concerned with two families of spaces of differential forms on
a compact Sasakian manifold,
namely, $\Omega^{p,\nu}_{d,i_\xi,\epseta\delta}(M)$ and
$\Omega^{p,\nu}_{\delta,\epseta,i_\xi d}(M)$, where $0\le p\le 2n+1$ and
$\nu$ are non-negative real numbers. It follows from~\eqref{deltal} that the
spaces of both families are included in  $\Omega^*_{\lie_\xi} (M)$.
Let us show that these two families are related to each other by the Hodge
star operator $*$.

\begin{proposition}
	\label{dual}
	Let $M$ be a compact Sasakian manifold of dimension $2n+1$. Then for any $p$-form $\omega$,
	we have
	  $\omega\in \Omega^{p,\nu}_{d,i_\xi, \epseta\delta}(M)$ if and
only if $*\omega \in \Omega^{2n+1-p,\nu}_{\delta,\epseta,\ixi d}(M)$.
\end{proposition}
\begin{proof}
It is easy to check that for any $p$-form $\omega$, one has
$\delta * \omega = (-1)^{p+1} *d\omega$, $\epseta *\omega = (-1)^{p+1}
*\ixi\omega$.  Hence $\ixi d * \omega =-  * \epseta\delta \omega$. Moreover, and
 $\lap  * \omega =* \lap \omega $.
\end{proof}

Now we show that the spaces of harmonic forms on
 a compact Sasakian manifold are included in
the above families of subspaces.
\begin{proposition}
    \label{characterization}
Let $M$ be a compact Sasakian manifold of dimension $2n+1$.
\begin{enumerate}[(i)]
    \item For $p\le n$, we have
	    \begin{equation*}
	    	    \Omega^p_{\lap}\left( M \right) =
	    	    \Omega^{p,0}_{ d,\ixi,\epseta\delta}\left( M
	        \right)\mbox{\  and\ \ \  }\Omega^{p, 0}_{ \delta, \epseta,\ixi
		d}
		    (M)=0.
	    	    \end{equation*}
	    \item For $p\ge n+1$, we have
		    \begin{equation*}
		    		    \Omega^p_{\lap}\left( M \right) =
			            \Omega^{p,0}_{ \delta, \epseta, \ixi
				    d}\left( M
			            \right)\mbox{\  and\ \ \  }
			    	{\Omega^{p,0}_{d,\ixi,\epseta\delta}(M) =0}.
			\end{equation*}
	    \end{enumerate}
\end{proposition}
\begin{proof}
It is obvious that $\Omega^{p,0}_{ d,\ixi,\epseta\delta}\left( M
        \right)\subset \Omega^p_{\lap}\left( M \right)$. Let $\omega \in
        \Omega^{p}_\lap\left( M \right)$ with $p\le n$. Since $\omega$
        is harmonic, we have
            $d\omega = 0$ and  $\epseta \delta \omega  = 0$.
        Moreover, by Proposition~\ref{tachi} we have that $\ixi \omega =0$,
        since $p\le n$.
        Thus  $\Omega^{p,0}_{d,\ixi,\epseta \delta}(M) =
        \Omega^p_\lap(M)$.

        In order to prove that $\Omega^{p,0}_{ \delta, \epseta ,
        \ixi d}(M) =0$, notice that
\begin{align*}
    \Omega^{p,0}_{\delta,\epseta ,\ixi d}(M) \subset
    \Omega^{p}_\lap(M) = \Omega^{p,0}_{d,\ixi,\epseta \delta}(M)
\end{align*}
implies
$\epseta \omega =0$ and $\ixi \omega=0$ for any  $\omega\in
\Omega^{p,0}_{\delta,\epseta ,\ixi d}(M)$. Therefore using \eqref{lambdal}, we
get
\begin{equation*}
\omega = \ixi \epseta \omega + \epseta \ixi \omega = 0.
\end{equation*}
	Hence $\Omega^{p,0}_{\delta,\epseta ,\ixi d}(M) =0$.
The second part of the proposition can be proved by duality considerations, using Proposition~\ref{dual}.
\end{proof}
\begin{proposition}
    \label{propddelta}
    Let $M$ be a compact Sasakian manifold.
    \begin{enumerate}[(i)]
        \item If $\omega \in \Omega^{p, 4\nu}_{\delta,
            \epseta ,\ixi d}(M)$ then $d\omega \in \Omega^{p+1,
            4\nu}_{ d, \ixi, \epseta \delta}(M)$.
            If moreover $\nu\not=0$, then $d\omega \not=0$.
	    \vspace{0ex}
        \item If $\omega \in \Omega^{p,4\nu}_{ d,\ixi,
            \epseta \delta}(M)$, then $\delta \omega \in \Omega^{p-1,
            4\nu}_{\delta, \epseta , \ixi d}(M)$.
            If moreover $\nu\not=0$, then $\delta\omega \not=0$.
    \end{enumerate}
\end{proposition}
\begin{proof}
	By duality established in Proposition~\ref{dual}, it is enough to prove just $(i)$. Let $\omega\in\Omega^{p, 4\nu}_{\delta,
            \epseta ,\ixi d} (M)$. Then
            \begin{align*}
 d\left( d\omega \right) &= 0, &
\ixi d \omega &= 0, & \lap d \omega& = d\lap \omega = 4\nu d\omega.
            \end{align*}
It is left to show that $\epseta \delta d \omega =0$.
Since $\delta \omega =0$ and $\epseta \omega =0$, we get
\begin{align*}
    \epseta \delta d \omega = \epseta \delta d \omega + \epseta  d \delta \omega
    = \epseta \lap \omega =
 4\nu \epseta \omega = 0.
\end{align*}
If moreover $\nu \not=0$, then
$\delta d \omega = \lap \omega = 4\nu \omega \not=0$.
            Thus also  $d\omega \not=0$.

\end{proof}
Proposition~\ref{propddelta} shows that for $\nu\not=0$ we have
the pair of isomorphisms
\begin{equation}
    \label{isomorphismddelta}
    \xymatrix{ \Omega^{p, 4\nu}_{ \delta, \epseta , \ixi d}\left( M
\right)\ar@<2pt>[r]^d  &  \Omega^{p+1,
4\nu}_{ d, \ixi, \epseta \delta}\left( M \right)\ar@<2pt>[l]^\delta},
\end{equation}
    for any $0\le p\le 2n$.
\begin{theorem}
    \label{complicated}
    Let $M$ be a compact Sasakian manifold of dimension $2n+1$.
    \begin{enumerate}[(i)]
            \vspace{0.5ex}
        \item If $\omega\in
    \Omega^{p,4\nu}_{\delta,\epseta ,\ixi d}\left( M
    \right)$ then $\ixi\omega\in
    \Omega^{p-1,4\left( \nu + p - n-1 \right)}_{d,\ixi, \epseta  \delta}\left(
            M \right)$.
            \vspace{0.9ex}
        \item If $\omega \in \Omega^{p, 4\nu}_{ d, \ixi,
            \epseta \delta}\left( M \right)$ then $\epseta \omega \in
            \Omega^{p+1, 4\left( \nu -p + n \right)}_{
            \delta, \epseta , \ixi d}\left( M \right)$.
                \end{enumerate}
     \end{theorem}
\begin{proof}
We will prove just $(i)$ as $(ii)$ can be obtained from $(i)$ by using Hodge
duality and Proposition~\ref{dual}.
    Let $\omega\in \Omega^{p,4\nu}_{ \delta,\epseta , \ixi d}\left(
    M \right)$.
  We write $\nu'$ for $\nu + p -n -1$.
  We have to show that $i_\xi \omega \in \Omega^{p-1,4\nu'}_{d,i_\xi,
  \epseta\delta}(M)$.
First of all, from \eqref{deltal},
we get
\begin{align}
    \label{obvious}
 d\ixi \omega& = \lie_\xi \omega-\ixi
    d\omega =  -\left\{ \epseta,\delta \right\}\omega=  0 .
\end{align}
Next, it is obvious that $\ixi\ixi \omega =0$.
Moreover, by using \eqref{deltal}, \eqref{lambdal} and \eqref{obvious}  we get
\begin{align*}
    \epseta  \delta \ixi \omega = -\lie_\xi \ixi \omega - \delta \epseta \ixi
    \omega = -\ixi d \ixi \omega - \delta \left(\omega -  \ixi \epseta \omega\right) = 0.
\end{align*}
Thus $\ixi\omega \in \Omega^{p-1}_{d,\ixi,\epseta \delta}(M)$.
It is left to prove that $\lap i_\xi \omega = 4\nu' i_\xi \omega$.
From the equation~\eqref{laplambda} for $\left[ \lap,\ixi \right]$, using that
$\delta \omega =0$, we get
\begin{equation*}
\lap i_\xi \omega - i_\xi \lap \omega = -2\delta i_\phi \omega - 4(n-p+1)
i_\xi\omega.
\end{equation*}
Since $\lap\omega =4\nu\omega$ and $\nu-n+p-1 = \nu'$, we get
\begin{equation}
	\label{lapixiomega}
	\lap \ixi \omega = 4\nu' \ixi \omega - 2 \delta i_\phi \omega.
\end{equation}
To get that $\delta i_\phi \omega=0$, we proceed as follows. First we
apply $d\delta$ to \eqref{lapixiomega}. Using that $d\delta$ commutes
with $\lap$, we get
\begin{equation*}
	\lap d\delta i_\xi \omega = 4\nu' d\delta i_\xi \omega.
\end{equation*}
As $d\ixi\omega=0$ by~\eqref{obvious}, we have $d\delta \ixi\omega = \lap
\ixi\omega$. Therefore
\begin{equation}
	\label{laplapixiomega1}
	\lap \lap \ixi \omega = 4\nu' \lap \ixi\omega.
\end{equation}
Comparing \eqref{lapixiomega} and \eqref{laplapixiomega1}, we see that
$\lap\delta i_\phi \omega = 0$, thence also
\begin{equation}
	\label{deltalapiphiomega}
\delta \lap i_\phi \omega=0.
\end{equation}
	Further, using the formula~\eqref{lapiphi} for $\left[ \lap, i_\phi \right]$, we get
\begin{align}
	\label{lapiphiomega}
    \lap i_\phi \omega &= i_\phi \lap \omega - 2 \lie_\xi\omega + 2 \ixi d \omega - 2
 \epseta \delta \omega  = 4\nu i_\phi \omega,
\end{align}
since $\omega\in \Omega^{p,4\nu}_{\delta,\epseta,\ixi d}(M)$ and hence $\lie_\xi \omega=0$.
From ~\eqref{deltalapiphiomega} and \eqref{lapiphiomega} it follows that
 $ 4\nu \delta i_\phi\omega= \delta \lap i_\phi \omega = 0$.
If $\nu\not=0$, this implies that $\delta i_\phi \omega =0$. On the other hand
if $\nu=0$, then
$\omega$ is harmonic. Thus by Proposition~\ref{tachi} the form $i_\phi \omega$
is also harmonic. Hence $\delta i_\phi \omega =0$.
Therefore from \eqref{lapixiomega} we finally get that $\lap \ixi\omega =
4\nu'\ixi\omega$.
\end{proof}
Since $\left\{ \epseta ,\ixi \right\} =\Id$, we get from Theorem~\ref{complicated}
that $\epseta $ and $\ixi$ induce the pair of inverse isomorphisms
\begin{equation}
    \label{isomorphismlambdal}
    \xymatrix{ \Omega^{p, 4\nu}_{ \delta, \epseta , \ixi d}\left( M
\right)\ar@<2pt>[r]^\ixi  &  \Omega^{p-1,
4\nu'}_{ d, \ixi, \epseta \delta}\left( M
\right)\ar@<2pt>[l]^\epseta },
\end{equation}
    where $\nu' = \nu + p -n -1$ and $1\le p\le 2n+1$.
In fact, if $\omega$ is on the left hand side, then $\epseta \omega =0$ and
$\epseta \ixi
\omega = \omega - \ixi \epseta  \omega =\omega$. Similarly, if $\omega$ is on the
right hand side, then $\ixi \omega =0$ and $\ixi \epseta  \omega = \omega$.

\begin{corollary}
    \label{LLambda}
    Let $M$ be a compact Sasakian manifold of dimension $2n +1$.
    Then
    for $\nu\not=0$ and $p\le 2n-1$
\begin{equation}
    \label{isomorphismLLambda}
    \xymatrix{ \Omega^{p, 4\nu}_{\delta, \epseta , \ixi d}\left( M
\right)\ar@<2pt>[r]^-L &  \Omega^{p+2,
4(\nu-p-1+n)}_{ \delta,\epseta ,\ixi d}\left( M \right)\ar@<2pt>[l]^-\Lambda},
\end{equation}
is a pair of isomorphisms such that $\Lambda L = \nu\Id$ and $L\Lambda = \nu\Id$.
Moreover, for $\nu\not=p-n$
\begin{equation}
    \label{isomorphismLambdaL}
    \xymatrix{ \Omega^{p, 4\nu}_{ d, \ixi, \epseta \delta}\left( M
\right)\ar@<2pt>[r]^-L &  \Omega^{p+2,
4(\nu-p+n)}_{ d, \ixi, \epseta \delta}\left( M \right)\ar@<2pt>[l]^-\Lambda},
\end{equation}
is a pair of isomorphisms such that $L\Lambda = (\nu-p+n)\Id$ and $\Lambda L=
(\nu-p + n)\Id$.
\end{corollary}
\begin{proof}
Notice that from \eqref{dl}, we have that for every $\omega \in \Omega^{p,
4\nu}_{\delta, \epseta , \ixi d}\left( M \right)$
\begin{equation}
    \label{Lomega}
2L\omega =\epseta d\omega.
\end{equation}
    Similarly from ~\eqref{deltalambda}, for every $\omega \in \Omega^{p+2, 4\left( \nu -p -1+ n
\right)}_{\delta, \epseta , \ixi d}(M)$, we get
\begin{equation}
    \label{Lambdaomega}
2\Lambda \omega =\delta\ixi \omega.
\end{equation}
Therefore, using the isomorphisms \eqref{isomorphismddelta} and \eqref{isomorphismlambdal},
we can construct the diagram
\begin{equation*}
\xymatrix{
 \Omega^{p,4\nu}_{\delta,\epseta ,\ixi d}(M)
 \ar@<2pt>[r]^-d \ar@/^5ex/[rr]^-{2L} &
 \Omega^{p+1,4\nu}_{d,\ixi,\epseta \delta}(M)
\ar@<2pt>[r]^-\epseta  \ar@<2pt>[l]^-\delta &
\Omega^{p+2,4(\nu-p-1+n)}_{\delta,\epseta ,\ixi d}(M).
\ar@<2pt>[l]^-\ixi \ar@/^5ex/[ll]^-{2\Lambda}
}
\end{equation*}
This shows that $L$ and $\Lambda$ induce isomorphisms between the
spaces in~\eqref{isomorphismLLambda}.
 For every
$\alpha \in \Omega^{p, 4\nu}_{ \delta, \epseta ,\ixi d}(M)$ and every
$\beta \in \Omega^{p+2, 4(\nu-p-1 + n)}_{ \delta, \epseta ,\ixi d}(M)$,
we have
\begin{align*}
    \Lambda L \alpha &=\frac14 \delta \ixi \epseta  d \alpha =\frac14 \delta d
    \alpha =\frac14 \lap
    \alpha = \nu \alpha\\
    L\Lambda \beta &=\frac14 \epseta  d \delta \ixi \beta =\frac14 \epseta  \lap \ixi \beta =
  \nu \epseta \ixi \beta =  \nu \beta.
\end{align*}
The second part of the corollary is proved by the same line of reasoning from the
diagram
\begin{equation*}
\xymatrix{
\Omega^{p,4\nu}_{d,\ixi,\epseta \delta}(M)
\ar@<2pt>[r]^-\epseta  \ar@/^5ex/[rr]^-{2L} &
\Omega^{p+1,4(\nu-p + n)}_{\delta,\epseta ,\ixi d}(M)
\ar@<2pt>[r]^-d \ar@<2pt>[l]^-\ixi &
\Omega^{p+2,4(\nu-p + n)}_{d,\ixi,\epseta \delta}(M)
\ar@<2pt>[l]^-\delta \ar@/^5ex/[ll]^-{2\Lambda}.
}
\end{equation*}
\end{proof}
Now, we are prepared to prove the Hard Lefschetz Theorem at the level of harmonic forms.
\begin{theorem}
    \label{lefschetz}
Let $M$ a compact Sasakian manifold of dimension $2n+1$ and $p\le n$.
Then the map
$$
\epseta L^{n-p} \colon \Omega^p\left( M \right)\to \Omega^{2n+1-p}\left( M
\right)
$$
induces an isomorphism $F_p\colon \Omega^p_\lap\left( M \right)\to \Omega^{2n+1-p}_\lap\left( M \right)$.
Similarly, the map
$$
 \Lambda^{n-p}\ixi \colon \Omega^{2n + 1-p}(M)\to \Omega^p(M)
$$
induces an isomorphism $G_p\colon \Omega^{2n+1-p}_\lap\left( M \right)\to\Omega^p_\lap\left( M \right)$.
Moreover,
\begin{align*}
    F_p G_p & =  (n-p)!^2\cdot I , & G_p F_p & = (n-p)!^2\cdot I.
\end{align*}

\end{theorem}
\begin{proof}
    We have by Proposition~\ref{characterization} that $\Omega^p_\lap\left(
    M \right) = \Omega^{p,0}_{d,\ixi,\epseta \delta}\left( M
    \right)$ and  $\Omega^{2n+1-p}_\lap\left( M \right)=
    \Omega^{2n+1-p,0}_{ \delta, \epseta , \ixi d}(M)$.
    Let us define the numbers $\nu_{p,k}$ by
    \begin{align}
        \label{nuprod}\nu_{p,k}&:= k(n-p -k +1),\ k\in \mathbb{Z}.
    \end{align}
    To make the notation less heavy, we will write $\nu_k$ instead of
    $\nu_{p,k}$ along the proof.
    It is easy to check that $\nu_{k+1} = \nu_k -(p+2k)+ n$,
    and that  $\nu_k\not=0$ for $1\le k\le n-p$. Thus by
    Corollary~\ref{LLambda}, the operators $L$ and $\Lambda$ induce
    isomorphisms between $\Omega^{p + 2k, 4\nu_k}_{ d, \ixi,
    \epseta \delta}(M)$ and $\Omega^{p + 2(k+1), 4\nu_{k+1}}_{d,\ixi,
    \epseta \delta}(M)$.
    Moreover, for $\omega\in \Omega^{p+2k, 4\nu_k}_{ d,\ixi,
    \epseta \delta}(M)$, it holds
    \begin{equation}
        \label{LambdaLomega}
        \Lambda L \omega = \nu_{k+1}\omega.
    \end{equation}
        Note that $\nu_0=0$ and $\nu_{n-p} = n-p$.
    Then we get the chain of isomorphisms
    \begin{equation}
        \label{chain}
\xymatrix{
\Omega^{p,0}_{d, \ixi, \epseta \delta}(M)\ar[r]^-{L} &
\Omega^{p+2,4\nu_1}_{ d, \ixi, \epseta \delta}(M) \ar[r]^-L &
\dots \ar[r]^-L &
\Omega^{2n-p,4( n-p)}_{ d,\ixi, \epseta \delta}(M)
}.
    \end{equation}
    Thus
    $L^{n-p}$ induces an isomorphism between the vector spaces $\Omega^{p,0}_{ d,
    \ixi, \epseta \delta}(M)$ and $\Omega^{2n-p, 4\left( n-p \right)}_{ d,\ixi,
    \epseta \delta}(M)$.
 From \eqref{isomorphismlambdal}, we have that $\epseta $ and $\ixi$ induce
    two mutually inverse maps between $\Omega^{2n-p,4(n-p)}_{ d,
    \ixi, \epseta \delta}\left( M \right) $ and
    $\Omega^{2n-p+1,0}_{\delta,\epseta ,\ixi d}(M) =
    \Omega^{2n-p+1}_\lap(M)$.
    This shows, that $F_p$ is an isomorphism between $\Omega^p_\lap\left(
    M \right)$ and $\Omega^{2n-p+1}_\lap(M)$. Similarly, one can check that
    also $G_p\colon \Omega^{2n-p+1}_\lap(
    M)\to \Omega^{p}_\lap(M)$ is an isomorphism.
    Iterating \eqref{LambdaLomega}, we see that for all $\omega\in
    \Omega^p_\lap(M)$
    \begin{align*}
        G_p F_p \omega& = \Lambda^{n-p}L^{n-p}\omega = \left(
        \prod_{k=1}^{n-p}\nu_k \right)\omega = \left(
        \prod_{k=1}^{n-p} k(n-p -k + 1)
        \right)\omega\\& = (n-p)!^2 \omega.
    \end{align*}
    Similarly, for $\omega\in\Omega^{2n-p+1}_\lap(M)$, we have $F_pG_p\omega
    =  (n-p)!^2 \omega$.
\end{proof}
    From~\eqref{chain} it follows that the
    spaces $\Omega^{p+2k,4\nu_{p,k}}_{d,\ixi, \epseta \delta}(M)$, where
    $\nu_{p,k}$ are defined by~\eqref{nuprod},
are isomorphic to $\Omega^p_\lap(M)$ for all $1\le k\le n-p$. Thus, if
$H^p\left( M \right)\not=0$, the same is true for
$\Omega^{p+2k,4\nu_{p,k}}_{d,\ixi,\epseta \delta} (M)$.
The following proposition shows that the only pairs $(p,\nu)$ such that
$\Omega^{p,4\nu}_{d,\ixi,\epseta \delta}(M)\not=0$ are necessarily of the form $(q+2k,
\nu_{q,k})$.

    \begin{proposition}
    \label{vanishing}Let $M$ be a compact Sasakian manifold of dimension
    $2n+1$.
    \begin{enumerate}[(i)]
        \item If
            $\Omega^{p,4\nu}_{d,\ixi,\epseta \delta}(M)\not=0$
            then $\nu = \nu_{p-2k,k}$ for some integer $k$ such that
	    \begin{equation*}
	    	   \max\{0, (p-n)/2\}\le k \le p/2.
	   \end{equation*}
        \item If $\Omega^{p,4\nu}_{\delta,\epseta ,\ixi
            d}(M)\not=0$ then $\nu = \nu_{p+1-2k,k}$ for some
            integer $k$ such that
	    \begin{equation*}
	    	    \max\{ 0, (p+1-n)/2 \} \le k\le
	                (p+1)/2.
		\end{equation*}
    \end{enumerate}
\end{proposition}
\begin{proof}
Note that the second part of the proposition follows from the first part and the
isomorphism~\eqref{isomorphismddelta}.
We will prove the first part by induction on $p$. Suppose
$p=0$ and $\Omega^{0,4\nu}_{d,\ixi, \epseta \delta}(M)\not=0$.
Let $f \in \Omega^{0,4\nu}_{d,\ixi,\epseta \delta}(M)$, $f\not=0$. Then
$d f =0$ and thus $f$ is a  harmonic function. Since $ 4\nu$ is a $\lap$-eigenvalue of
$f$, we get that necessarily $\nu = 0$. Thus $\nu=\nu_{0,0}$.

Suppose we proved the claim for all $p'<p$ and $\Omega^{p,4\nu}_{
d,\ixi,\epseta \delta }(M)\not=0$.
If $\nu=0$, then by Proposition~\ref{characterization}
we get that $p\le n$. Therefore $0$ is in the allowed range of values for
$k$, and  we are done since $\nu_{p,0} = 0$.
Now consider the case $\nu\not=0$. From Corollary~\ref{LLambda}, we get that the
spaces $\Omega^{p-2, 4(\nu + p-2 - n)}_{ d,\ixi,\epseta \delta}(M)$ and
$\Omega^{p, 4\nu}_{d,\ixi,  \epseta \delta}(M)$ are isomorphic. From the
induction assumption, we get that
$$
\nu + p -2 - n = \nu_{(p-2)-2k',k'}
$$
for some $k'$ between $\max\{0, (p-2-n)/2\}$ and $(p-2)/2$.
Define $k=k'+1$. Then
\begin{align*}
    \nu &= \nu_{p-2k,k-1} -p + 2 + n\\
    &= (k-1)(n-p +k + 2) +n-p+2 \\
    &
    = k(n -p +k +1)\\& = \nu_{p-2k,k}.
\end{align*}
Moreover $k$ lies between $\max\{ 1, (p-n)/2\}$ and $p/2$. This completes the
induction argument.
\end{proof}

\begin{figure}
    \label{plot}
\xygraph{
!{<0ex,0ex>;<0.96cm,0cm>:<0cm,0.96cm>::}
!{(0,0)} :@{.}_(0){0}|(0){\gb}_(1){1}[r]|(1){\gb}
:@{.}_(1){2}|(1){\gb}[r]
:@{.}_(1){3}|(1){\gb}[r]
:@{.}_(1){4}|(1){\gb}[r]
:@{.}_(1){5}|(1){\gb}[r]
:@{-}_(1){6}|(1){\gc}[r]
:@{.}_(1){7}|(1){\gc}[r]
:@{.}_(1){8}|(1){\gc}[r]
:@{.}_(1){9}|(1){\gc}[r]
:@{.}_(1){10}|(1){\gc}[r]
:@{.}_(1){11}|(1){\gc}[r]
:@{.>}_(0.8){p}[r]
[llllllllllll]:@{.}^(1){1}[u]
:@{.}^(1){2}[u]
:@{.}^(1){3}[u]
:@{.}^(1){4}[u]
:@{.}^(1){5}[u]
:@{.}^(1){6}[u]
:@{.}^(1){7}[u]
:@{.}^(1){8}[u]
:@{.}^(1){9}[u]
:@{.}[u]
:@{.>}^(.51){\nu}[u]
[ddddddddddd]:@{-}|(1){\gc}[uuuuur]
:@{-}|(1){\gb}[r]
:@{-}|(1){\gc}[uuur]
:@{-}|(1){\gb}[r]
:@{-}|(1){\gc}[ur]
:@{-}|(1){\gb}[r]
:@{-}|(1){\gc}[dr]
:@{-}|(1){\gb}[r]
:@{-}|(1){\gc}[dddr]
:@{-}|(1){\gb}[r]
:@{-}[dddddr]
[llllllllll]
:@{-}|(1){\gc}[uuuur]
:@{-}|(1){\gb}[r]
:@{-}|(1){\gc}[uur]
:@{-}|(1){\gb}[r]
:@{-}|(1){\gc}[r]
:@{-}|(1){\gb}[r]
:@{-}|(1){\gc}[ddr]
:@{-}|(1){\gb}[r]
:@{-}|(1){\gc}[ddddr]
[llllllll]
:@{-}|(1){\gc}[uuur]
:@{-}|(1){\gb}[r]
:@{-}|(1){\gc}[ur]
:@{-}|(1){\gb}[r]
:@{-}|(1){\gc}[dr]
:@{-}|(1){\gb}[r]
:@{-}|(1){\gc}[dddr]
[llllll]
:@{-}|(1){\gc}[uur]
:@{-}|(1){\gb}[r]
:@{-}|(1){\gc}[r]
:@{-}|(1){\gb}[r]
:@{-}|(1){\gc}[ddr]
[llll]
:@{-}|(1){\gc}[ur]
:@{-}|(1){\gb}[r]
:@{-}|(1){\gc}[dr]
}
\caption{}
\end{figure}
In Figure~1, we give an illustration of what happens for  a compact
Sasakian manifold  $M$ of dimension $11$.
In the picture the circles mark the places $(p,\nu)$ such that
$\Omega^{p, 4\nu}_{ \delta, \epseta ,\ixi d}(M)$ can be non-zero.  Similarly, the squares mark the points $(p,\nu)$ such that the spaces
$\Omega^{p,4\nu}_{d,\ixi, \epseta \delta}(M)$ can be non-zero.
The horizontal segments with the circle at the left
edge and the square at the right edge represent the isomorphisms of
the type \eqref{isomorphismddelta}, and all
the other segments correspond to the isomorphisms of the type~\eqref{isomorphismlambdal}.
Thus, we can see that if we start with a harmonic $p$-form $\omega$ for $p\le n$
and
move it along the segments representing
the isomorphisms of the types \eqref{isomorphismddelta} and
\eqref{isomorphismlambdal}, we get in the intermediate steps $\lap$-eigenforms with
non-zero eigenvalues, and we will eventually end up with the  harmonic
$(2n+1-p)$-form $F_p \omega$.

\section{Hard Lefschetz isomorphism in de Rham cohomology}
\label{nometric}

Let $M$ be a compact Sasakian manifold of dimension $2n+1$. Let us denote by $\pr$ the
orthogonal projection from $\Omega^*(M)$ onto $\Omega^*_\lap(M)$. Then we can
define
\begin{equation}
	\label{lefp}
	\begin{aligned}
		\lef_p\colon 	H^p(M) &\to H^{2n+1-p}(M)\\
	\left[\, \beta \,\right] &\mapsto \left[\, F_p\,\pr\,  \beta \,\right].
\end{aligned}
\end{equation}
	Due to Hodge theory and Theorem~\ref{lefschetz} the map $\lef_p$ is an
isomorphism.

Suppose $\left( \eta, g' \right)$ is another Sasakian structure on $M$ with
the same contact form~$\eta$.
Denote by $\lap'$ the corresponding Laplacian.
For a closed $p$-form $\beta$
it can happen that
\begin{equation*}
	\pr \beta \not= \prprime \beta.
\end{equation*}
Thus it is not \emph{a priori} clear whether $\lef_p$ depends on the full Sasakian structure
or just on its contact form. Note that this is rather
different from the K\"ahler case, where it is obvious from the definition of the
Lefschetz map that it is fully determined by the symplectic structure and does
not depend on chosen K\"ahler metric.

The aim of this section is to show that $\lef_p$ is uniquely determined by the
contact structure of $M$. Note that from Proposition~\ref{tachi}
it follows that for any Sasakian metric $g'$  on $(M,\eta)$ and any closed $p$-form
$\beta$ with $p\le n$, we have
\begin{equation*}
	\prprime \beta \in \Omega^p_{ \ixi,d}(M).	
\end{equation*}
Moreover, since $d$ commutes with $L$, by applying~\eqref{dl} and
Theorem~\ref{lefschetz} to $(M,g')$ we get
\begin{equation}
	\label{Lkprprimebeta}
	L^{n-p+1} \prprime \beta =\frac12 d\left( \epseta L^{n-p} \prprime \beta
	\right) = \frac12 d\left( F_p \prprime \beta \right) =0.	
\end{equation}
Indeed, we will prove that given an arbitrary $p$-form $\gamma$ with $p\le n$ such that
\begin{align}
	\label{conditionsbeta}
	\ixi \gamma &= 0, & d\gamma &= 0, &
	L^{n-p+1}\gamma &= 0 ,
\end{align}
the cohomology classes of $\epseta L^{n-p}\gamma$ and of $\epseta L^{n-p}\pr\gamma$ are equal.
As a consequence
we will get that
\begin{equation*}
	\left[ \epseta L^{n-p} \prprime \beta \right] = \left[ \epseta L^{n-p}\pr\prprime
	\beta \right] = \left[ \epseta L^{n-p} \pr \beta \right],
\end{equation*}
since $\prprime \beta$ satisfies~\eqref{conditionsbeta} and
$\pr\prprime\beta=\pr\beta$ as $[\prprime\beta]=[\beta]$.

We start by recalling some facts on Hodge theory for compact manifolds. Let $M$ be a compact
Riemannian manifold. Denote by  $G$ the Green operator for the de Rham complex
$\Omega^*(M)$ of $M$. Then
\begin{align}
	\label{greenoperator}
\Id - \lap G &= \pr, & \Id - G\lap & = \pr, & dG &= Gd, & \delta G & = G\delta.
\end{align}

\begin{proposition}
	\label{greenliexi}
	Let $(M,g)$ be a compact Riemannian manifold and $\xi$ a Killing vector
	field on $M$. Then  $\lie_\xi$ commutes with the Green operator
	$G$.
\end{proposition}
\begin{proof}

	By Theorem~3.7.1 in \cite{goldberg}, we have that $\lie_\xi\omega =0$ for every harmonic form
	$\omega$. Thus $\lie_\xi \pr = 0$.
Let us show that $\pr \lie_\xi=0$.
	Using \eqref{deltal} and that $d^* = \delta$, $\ixi^*=\epseta $, we get
	\begin{equation*}
		\lie_\xi^* = \left\{ d, \ixi \right\}^* = \left\{ \epseta , \delta
		\right\} = - \lie_\xi.
	\end{equation*}
	Hence for any $\beta\in \Omega^k(M)$ and $\omega \in
	\Omega^k_{\lap}(M)$  we have
	\begin{align*}
\left( \lie_\xi \beta , \omega \right)& = - \left( \beta, \lie_\xi \omega \right)
=0,
	\end{align*}
	where $(\ ,\ )$ denotes  the usual global scalar product
	of
		differential forms on a compact Riemannian manifold.
	This shows that $\lie_\xi \beta$ is orthogonal to the subspace
	$\Omega^k_{\lap}(M)$ of $\Omega^k(M)$ and thus  $\pr \lie_\xi \beta =0$.
	Now, from~\eqref{greenoperator} we get
	\begin{align*}
&\lie_\xi - \lie_\xi \lap G =\lie_\xi \pr= 0, && \lie_\xi - G \lap \lie_\xi =\pr
\lie_\xi =0.
	\end{align*}
	We know from Theorem~\ref{commutators} that  $\lie_\xi$ commutes with
	$d$ and $\delta$, and thus also with
	$\lap = \left\{ d,\delta \right\}$.
	Hence
\begin{align*}
G \lie_\xi = G\lie_\xi \lap G = G \lap \lie_\xi G = \lie_\xi G.
\end{align*}
\end{proof}
Let us introduce the auxiliary map
\begin{align*}
	\aux_p \colon \Omega^{p-1}(M) & \to \Omega^{2n-p+1}(M)\\
	\alpha & \mapsto (n-p+1) L^{n-p} di_\phi d \alpha + L^{n-p+1}\lap
	\alpha.
\end{align*}
We will  study the interplay between $\aux_p$ and $\delta$, $d$, $\lap$.
The following technical lemma is needed.

\begin{lemma}
	\label{deltaLk}
Let $M$ be a Sasakian manifold. Then for any $k\ge 1$, we have
\begin{align}\label{eqdeltaLk}
	\left[ \delta, L^k \right] & = -kL^{k-1}\lie_\phi + 2k\epseta L^{k-1}(n-\deg -
	(k-1)).
\end{align}
\end{lemma}
\begin{proof}
 For any two linear endomorphisms $a$ and $b$ of an arbitrary vector space
 $V$, we denote by $\ad_b(a)$ their
commutator $\left[ b,a \right]$.
It can be checked (e.g. by induction)
that
\begin{equation}
	\label{binomial}
	ab^k = \sum_{j=0}^k (-1)^j \binom{k}{j} b^{k-j} \ad_b^j (a).
\end{equation}
We will apply formula \eqref{binomial} to the case $a = \delta$ and $b = L$.
Dualizing~\eqref{dLambda}, since $i_\phi^* = -i_\phi$, we get
\begin{equation*}
	\ad_{L}(\delta) = \left[ L, \delta \right]= \left[ d, -i_\phi \right] -
	2 \epseta (n -\deg) = \lie_\phi - 2\epseta (n-\deg ).
\end{equation*}
Moreover, since $\lie_\phi$ is a derivation and $L= \eps_\Phi$, we get
\begin{align*}
\left[ \lie_\phi, L \right] = \eps_{\lie_\phi \Phi}= 0.
\end{align*}
Therefore
\begin{equation*}
	\ad^2_L(\delta) = \left[ L,\lie_\phi - 2\epseta n +  2\epseta\deg \right] = 2\epseta \left[ L,
	\deg
	\right] = -4  \epseta L.
\end{equation*}
Thus $\ad_L^j (\delta ) = 0$ for all $j\ge 3$.
Now the claim of the lemma follows from~\eqref{binomial}.
\end{proof}

\begin{theorem}
	\label{aux}
	Let $\left( M,\eta,g \right)$ be a compact Sasakian manifold of dimension $2n+1$,
	$1\le p\le n$, and
	$\alpha\in \Omega^{p-1}_{\lie_\xi, \delta}(M)$.
	Then $\aux_p \alpha$ is coclosed and
	\begin{equation}
		\label{num4}
	\lap \aux_p \alpha = \aux_p \lap \alpha.
	\end{equation}
\end{theorem}
\begin{proof}
	To check  that
	\begin{equation}
		\label{num3}
		\aux_p \alpha = (n-p+1) L^{n-p}di_\phi d\alpha + L^{n-p+1} \lap
		\alpha
		\end{equation}
	is coclosed, we will repeatedly use Lemma~\ref{deltaLk}.
We have
\begin{align}
	\label{deltaLkdiphidalpha0}
	\delta\left( L^{n-p} di_\phi d \alpha \right) &= \left[ \delta,
	L^{n-p} \right] d i_\phi d \alpha + L^{n-p} \delta d i_\phi d \alpha.
\end{align}
From Lemma~\ref{deltaLk} we get
\begin{align*}
	\left[ \delta, L^{n-p} \right]d i_\phi d \alpha =	-(n-p) L^{n-p-1} \lie_\phi d i_\phi d \alpha,
\end{align*}
since the second summand in \eqref{eqdeltaLk} vanishes in this case due to the
degree of $\alpha$. Moreover,  by using \eqref{liephisquare} we have
$\lie_\phi d i_\phi d \alpha=\lie_\phi^2 d\alpha = -2L d\lie_\xi \alpha=0$.
Therefore we get
\begin{equation}
	\label{num2}
	\left[ \delta, L^{n-p} \right] di_\phi d\alpha = 0.
\end{equation}
	Now
we compute the second summand of~\eqref{deltaLkdiphidalpha0}. We have
\begin{equation}
	\label{deltadiphidalpha}
\delta d i_\phi d\alpha = \lap i_\phi d\alpha - d \delta i_\phi d \alpha.
\end{equation}
From formula~\eqref{lapiphi} for $\left[ \lap,i_\phi \right]$ we get
\begin{equation}
	\label{lapiphidalpha}
	\begin{aligned}
	\lap i_\phi d \alpha & = - 2 \left( \lie_\xi - \ixi d + \epseta \delta \right) d\alpha + i_\phi d
\lap\alpha
\\&= - 2 \epseta \delta d\alpha + i_\phi d \lap \alpha\\& = -2 \epseta \lap\alpha + i_\phi d \lap
\alpha,
\end{aligned}
\end{equation}
since $\alpha\in \Omega^{p-1}_{\lie_\xi, \delta}(M)$
implies $\delta d \alpha = \lap \alpha$.
	Now we compute the second summand of~\eqref{deltadiphidalpha}. From ~\eqref{laplambda} we get
	\begin{equation}
		\label{num1}
		\begin{aligned}
		d\delta i_\phi d\alpha &= d \left[ \delta, i_\phi \right]
		d\alpha  + di_\phi
	\delta d\alpha
\\&= -\frac12 d\left[\lap,\ixi  \right]d\alpha - 2(n-(p-1))d \ixi d\alpha +
di_\phi \lap \alpha.
\end{aligned}
\end{equation}
		Note that $d\left[ \lap,\ixi \right] d \alpha=-\left[ \lap,\ixi \right] d^2 \alpha=0$, since
\begin{equation*}
	\left\{ d, \left[ \lap,\ixi \right] \right\} = \left\{ \left[ d,\lap
	\right], \ixi
	\right\} + \left[ \lap, \lie_\xi \right] = 0.	
\end{equation*}
Also $d\ixi d \alpha = -d^2
\ixi \alpha + d\lie_\xi\alpha =0$. Thus~\eqref{num1} becomes
\begin{equation}
	\label{ddeltaiphidalpha}
	d\delta i_\phi d\alpha = d i_\phi \lap\alpha.
\end{equation}
Substituting~\eqref{lapiphidalpha} and~\eqref{ddeltaiphidalpha}
in~\eqref{deltadiphidalpha}, we get
\begin{equation}
	\label{deltadiphidalpha2}
	\delta d i_\phi d\alpha = -2\epseta \lap \alpha + i_\phi d \lap \alpha - di_\phi
	\lap \alpha = -2\epseta \lap \alpha + \lie_\phi \lap\alpha.
\end{equation}
Thus, in view of~\eqref{num2}, the formula~\eqref{deltaLkdiphidalpha0} becomes
\begin{equation}
	\label{deltaLkdiphidalpha}
	\delta\left( L^{n-p}di_\phi d\alpha \right) =
	-2\epseta L^{n-p}\lap \alpha + L^{n-p}\lie_\phi \lap\alpha.
\end{equation}
Now we compute the  value of $\delta$ on
the second summand  of $\aux_p\alpha$.
 Since $\delta\lap
\alpha = \lap \delta\alpha = 0$,
from Lemma~\ref{deltaLk} it follows that
\begin{equation}
	\label{deltaLklapalpha}
\begin{aligned}
	\delta\left( L^{n-p+1}\lap\alpha \right) &= \left[ \delta, L^{n-p+1}
	\right]\lap\alpha\\
 &=
		-(n-p+1)L^{n-p} \lie_\phi \lap \alpha
		\\&\phantom{=\ }	+2(n-p+1) \epseta L^{n-p} \lap \alpha
		\\&= -(n-p+1)(L^{n-p} \lie_\phi \lap \alpha - 2 \epseta L^{n-p}
		\lap\alpha).
	\end{aligned}
\end{equation}
From~\eqref{num3}, \eqref{deltaLkdiphidalpha} and~\eqref{deltaLklapalpha}, we get that
$\delta \aux_p \alpha = 0$, that is $\aux_p\alpha$ is coclosed.

Now we will prove that  $\lap \aux_p\alpha  = \aux_p \lap \alpha$. Since $\aux_p\alpha $ is coclosed, we have
$\lap \aux_p\alpha = \delta d \aux_p \alpha$.
As $d$ commutes with $L$, we get
from~\eqref{num3} that
\begin{equation*}
	d\aux_p\alpha = L^{n-p+1}d\lap \alpha.
\end{equation*}
Thus
\begin{equation*}
\lap \aux_p \alpha = 	\delta d\aux_p \alpha = \left[ \delta, L^{n-p+1} \right] d\lap \alpha +
	L^{n-p+1} \delta d\lap \alpha.
\end{equation*}
By Lemma~\ref{deltaLk} we get
\begin{align*}
	\left[ \delta, L^{n-p+1} \right] d \lap \alpha =
	-(n-p+1)L^{n-p} \lie_\phi d \lap \alpha = (n-p+1) L^{n-p} d i_\phi d\lap \alpha.
\end{align*}
Next,
\begin{equation*}
	L^{n-p+1}\delta d \lap \alpha =L^{n-p+1} \lap \delta d\alpha =L^{n-p+1}
	\lap^2 \alpha.
\end{equation*}
Thus,
\begin{equation*}
	\lap \aux_p \alpha = (n-p+1) L^{n-p} di_\phi d \lap \alpha +
	L^{n-p+1} \lap^2 \alpha = \aux_p \lap \alpha.
\end{equation*}
\end{proof}
\begin{corollary}
	\label{auxalphazero}
	Let $\left( M, \eta,g \right)$ be a compact Sasakian manifold of dimension $2n+1$ and
	$\alpha\in \Omega^{p-1}_{\lie_\xi, \delta}(M)$ with $p\le n$,  such that
	${L^{n-p+1}d\lap \alpha = 0}$.
	Then $\aux_p \alpha = 0$.
\end{corollary}
\begin{proof}
It is easy to check that $\aux_p\alpha$ is closed.
It is also coclosed by Theorem~\ref{aux}.
Therefore $\aux_p\alpha$ is  harmonic.
Now we consider the form $\aux_p G\alpha$. It follows from~\eqref{greenoperator}
and
Proposition~\ref{greenliexi} that $G\alpha\in
\Omega^{p-1}_{\lie_\xi,\delta}(M)$. By Theorem~\ref{aux}, we have
\begin{equation*}
\lap \aux_p G \alpha = \aux_p \lap G \alpha = \aux_p \alpha - \aux_p \pr \alpha.
\end{equation*}
It is immediate from the definition of $\aux_p$ that $\aux_p\omega=0$ for any
harmonic $(p-1)$-form $\omega$, in particular for $\pr \alpha$.
Thus $\lap \aux_p G\alpha = \aux_p \alpha$. As $\aux_p \alpha$
is harmonic, we obtain
\begin{equation*}
0 = (\lap \aux_p \alpha, \aux_p G \alpha) = \left( \aux_p  \alpha, \lap
\aux_p G \alpha\right) = (\aux_p \alpha, \aux_p \alpha). 	
\end{equation*}
Thus $\aux_p \alpha = 0$.
\end{proof}

Now we will prove the main result of the article.

\begin{theorem}
	\label{main2}
	Let $\left( M,\eta,g \right)$ be
	a compact Sasakian manifold, $p\le n$, and  $\beta$ a closed $p$-form.
For any  $\beta'\in \left[ \beta \right]$ such that
\begin{equation}
	\label{num0}
	\begin{aligned}
			\ixi \beta' &= 0,  &
			L^{n-p+1}\beta' &= 0 ,
		\end{aligned}
	\end{equation}
	we have $\lef_p([\beta])=\left[ \epseta L^{n-p}\beta' \right]$. 	
In particular, the Lefschetz map $\lef_p$ does not depend on the choice
of a compatible Sasakian metric on $(M,\eta)$.
\end{theorem}
\begin{proof}
Let us define $\gamma = \delta G \left( \beta' - \pr \beta \right)$. Then,
since $d$ and $\delta$ commute with $G$ by~\eqref{greenoperator} and $\beta'$ is closed, we get
\begin{equation}
	\label{dgamma}
	\begin{aligned}
d\gamma &= G d\delta \left( \beta' -
\pr \beta
\right) \\& = G d\delta \beta' = G \lap \beta' = \beta' - \pr\beta' = \beta' - \pr
\beta.
\end{aligned}
\end{equation}
As $\lef_p([\beta])=[\epseta L^{n-p} \pr \beta]$, we have to prove
that $\epseta L^{n-p} \beta'$ and $\epseta L^{n-p} \pr \beta$ are in
the same cohomology class.
Since
\begin{equation*}
	\epseta L^{n-p} d \gamma = \epseta L^{n-p} \beta' - \epseta
	L^{n-p}\pr\beta,
\end{equation*}
it is enough to show that $\epseta L^{n-p} d \gamma$ is exact.
Now, as
\begin{equation*}
	d\left( \epseta L^{n-p}\gamma \right) = 2L^{n-p+1} \gamma - \epseta L^{n-p}d\gamma,
\end{equation*}
the form $\epseta L^{n-p}d\gamma$ is exact if and only if $L^{n-p+1} \gamma$ is exact.

Let $\alpha = G\gamma$, that is $\alpha =  \delta G^2 (\beta' - \pr\beta)$.
We will now check that $\alpha$ satisfies the hypotheses of
Corollary~\ref{auxalphazero}, namely that $\lie_\xi\alpha$, $ \delta \alpha$,
and
$L^{n-p+1}d\lap\alpha$ are zero.
Since $\lie_\xi$ commutes with $\delta$ by Theorem~\ref{commutators} and with
$G$ by Proposition~\ref{greenliexi}, we get
\begin{equation*}
	\lie_\xi \alpha = \delta G^2 \lie_\xi\left( \beta'-\pr\beta \right).
\end{equation*}
As $\pr\beta$ is harmonic and $\xi$ is Killing, we have $\lie_\xi \pr\beta =0$
by \cite[Theorem~3.7.1]{goldberg}. Moreover, $\beta'$ is
closed and $i_\xi\beta'=0$ by assumption. Thus $\lie_\xi \beta'=0$. Therefore we
get
that $\lie_\xi \alpha =0$.
As
$\alpha$ is in the image of $\delta$,
we also have $\delta \alpha =0$.
Now
\begin{equation}\label{lapalfagamma}
	\lap \alpha = \lap G\gamma = \gamma - \pr\gamma =\gamma,
\end{equation}
as $\gamma$ is coexact and this implies $\pr\gamma =0$.
Thus by~\eqref{num0} and~\eqref{dgamma}, we get
\begin{equation*}
	L^{n-p+1}d\lap \alpha = L^{n-p+1}d\gamma =
	L^{n-p+1} \beta' - L^{n-p+1} \pr \beta = - L^{n-p+1}\pr\beta.
\end{equation*}
By Theorem~\ref{lefschetz} we know that $\epseta L^{n-p}\pr\beta$ is harmonic, therefore
\begin{equation*}
	L^{n-p+1}\pr \beta = \frac 12 d\left( \epseta L^{n-p} \pr \beta \right) =0.
\end{equation*}
We conclude that $L^{n-p+1} d\lap \alpha =0$.
Thus all conditions of Corollary~\ref{auxalphazero} are satisfied for $\alpha$.
 Therefore $\aux_p\alpha=0$ and thus by \eqref{lapalfagamma} we get that
\begin{align*}
	L^{n-p+1}	\gamma& =L^{n-p+1} \lap \alpha = - (n-p+1) L^{n-p} d
	i_\phi d \alpha\\& =
	d\left(  -(n-p+1)L^{n-p} i_\phi d \alpha \right)
\end{align*}
is an exact form.
\end{proof}

\begin{corollary}
Let $(M, \eta)$ be a compact contact manifold of dimension $2n+1$. Suppose $g$ and
	$g'$ are two different Sasakian metrics on $M$, both compatible with $\eta$. Then, for any
	closed $p$-form $\beta$ with $p\le n$, it holds $\left[ \epseta L^{n-p} \pr\beta \right] =
	\left[ \epseta L^{n-p}\prprime \beta \right]$.
\end{corollary}
\begin{proof}
We obviously have $\prprime\beta\in \left[ \beta \right]$.
By Proposition~\ref{tachi} applied to the Sasakian manifold $(M,\eta, g')$, we
get
$\ixi \prprime \beta = 0$. Next by~\eqref{Lkprprimebeta} we have
$L^{n-p+1} \prprime \beta=0$.
Thus the form $\beta' = \prprime \beta$ satisfies all the conditions of
Theorem~\ref{main2} for the Sasakian manifold $(M,\eta, g)$.
Therefore $\left[ \epseta L^{n-p} \prprime \beta \right] =\lef_p([\beta])=
	\left[ \epseta L^{n-p}\pr \beta \right]$.
\end{proof}

\section{Lefschetz contact manifolds}

Let $(M,\eta)$ be a compact contact manifold of dimension $2n+1$
such that there exists a Sasakian metric on $M$ compatible with the contact
form $\eta$.
Then by choosing an arbitrary such metric
$g$, we can define the maps $\lef_p$ for $p\le n$ as
in~\eqref{lefp} and these maps are isomorphisms.

Note that there is no obvious way to define  similar maps between cohomology
spaces of a general contact manifold $(M,\eta)$ of dimension $2n+1$.
To introduce the notion of  Hard Lefschetz property for a contact manifold, we define
the \emph{Lefschetz relation} between cohomology groups $H^p(M)$ and
$H^{2n+1-p}(M)$ of $(M, \eta)$ to be
\begin{equation*}
	\rel_{\lef_p} = \left\{ \left( [ \beta ], [ \epseta L^{n-p}\beta
	] \right) \middle| \beta \in \Omega^p(M),\ d\beta = 0,\ \ixi\beta =0,\
	L^{n-p+1}\beta =0 \right\}.	
\end{equation*}
Thus if $(M,\eta)$ admits a compatible Sasakian metric,  from
Theorem~\ref{main2} it
follows that $\rel_{\lef_p}$ is the graph of the isomorphism $\lef_p$.
This justifies the following definition.
\begin{definition}
	We say that a compact contact manifold $(M,\eta)$ has the
	\emph{Hard Lefschetz property} if for every $p\le n$ the relation $\rel_{\lef_p}$ is the graph of an
	isomorphism $\lef_p: H^p(M)\longrightarrow H^{2n+1-p}(M)$.
	Such manifolds will be called \emph{Lefschetz contact manifolds}.
\end{definition}

There is a simple extension to Lefschetz contact manifolds of the well-known property that  the odd Betti numbers $b_{2k+1}$ ($0\leq 2k+1 \leq n$) of compact Sasakian manifolds are even (\cite{fujitani}).

\begin{theorem}
	\label{oddbettieven}
	Let $\left( M,\eta\right)$ be a Lefschetz contact manifold of dimension
$2n+1$. Then  for every  odd $p \le n$ the Betti number $b_{p}$ is
even.

\end{theorem}
\begin{proof}
Let $p\leq n$. Since $\lef_p$ is an isomorphism, using  Poincar\'{e} duality we can define a nondegenerate
bilinear form $B$ on the de Rham cohomology vector space
$H^{p}(M)$ by putting
$$B(x,x')=\int_M \lef_p(x)\smile x'.$$
Note that
\begin{equation}\label{bilinear}
B(x,x')=\int_M (\epseta L^{n-p}\omega)\wedge\omega'=\int_M \eta\wedge\Phi^{n-p}\wedge\omega\wedge\omega',
\end{equation}
where $\omega\in x$ and $\omega'\in x'$ are closed $p$-forms such that $\ixi\omega =\ixi\omega '=0$ and
	$L^{n-p+1}\omega =L^{n-p+1}\omega'=0$. Such $\omega$ and $\omega'$ always exist since $\rel_{\lef_p}$ is the graph of a map.
Now, \eqref{bilinear} implies that $B(x,x')=(-1)^pB(x',x)$.
It follows that, when $p$ is odd, the vector space $H^{p}(M)$ is even dimensional.
\end{proof}

It would be interesting to find a characterization of Lefschetz contact manifolds
in the spirit of~\cite{brylinski,mathieu}.
It would be also interesting to find explicit examples of Lefschetz contact manifolds  which
do not admit any Sasakian structure. We will address these matters in our future research.

{\bf Note added.} After the paper was submitted, Yi Lin has considered some of the
above questions in~\cite{lin}.
\bibliography{hard}
\bibliographystyle{amsplain}

\end{document}